\documentclass[reqno]{amsart}

\usepackage{amsmath}%
\usepackage{amsfonts}%
\usepackage{amssymb}%
\usepackage{graphicx}
\usepackage[hidelinks]{hyperref,xcolor}

\definecolor{gray}{rgb}{0.2,0.2,.2}
\hypersetup{
    colorlinks=true,
    citecolor=blue,
		linkcolor=gray
}

\newtheorem{theorem}{Theorem}
\theoremstyle{plain}

\numberwithin{equation}{section}

\newcommand{\EQ}[1]{\begin{equation} \ensuremath{#1} \end{equation}}

\newcommand{\sig}[1]{\ensuremath{\sigma(#1)}}
\newcommand{\bprod}[2]{\ensuremath{\prod_{#1}^{#2}}}
\newcommand{\bsum}[2]{\ensuremath{\sum_{#1}^{#2}}}

\begin{document}
\title[Partial Proof of Dris's Conjecture]{A Partial Proof of a Conjecture of Dris}
\author{Patrick A. Brown}
\address{Dr. Brown is not currently affiliated with any institution. Feel free to \newline
\indent email him regarding any questions on this paper.}
\email{PatrickBrown496@gmail.com}
\urladdr{}

\thanks{}

\date{January 2016}
\subjclass[2010]{11-04, 11A25}
\keywords{Odd Perfect Numbers, Dris's Conjecture}
\dedicatory{}

\begin{abstract}
Euler showed that if an odd perfect number $N$ exists, it must consist of two parts $N=q^k n^2$, with $q$ prime, $q \equiv k \equiv 1 \pmod{4}$, and gcd$(q,n)=1$. Dris conjectured in \cite{dris2008} that $q^k < n$. We first show that $q<n$ for all odd perfect numbers. Afterwards, we show $q^k < n$ holds in many cases.
\end{abstract}
\maketitle

\section{Introduction}

We define \sig{N} to be the sum of the positive divisors of $N$ and note the following properties of $\sigma$, which we will use freely:
\begin{enumerate}
\item $\sigma(p^b)= 1+p+p^2+$ \ldots $+p^b$ for powers of primes.
\item $\sigma(ab) = \sigma(a)\sigma(b)$, whenever gcd$(a,b)=1$.
\item $(\frac{p-1}{p}) \sigma(p^{2b}) < p^{2b}$. (Note that $(\frac{2}{3}) \sigma(p^{2b}) < p^{2b}$ works for any odd prime $p$.)
\end{enumerate}

We say $N$ is perfect when $\sigma(N) = 2N$. Euler showed that if an odd perfect number $N$ exists, then its factorization consists of two parts. A special prime $q$ appearing an odd number, say $k$ times, such that $q \equiv k \equiv 1 \pmod{4}$. The rest of the primes in the factorization appear an even number of times, which we represent as $n^2$. It is understood that gcd$(q,n)=1$. When written as $N=q^k n^2$ we say $N$ is written in Eulerian form. The condition gcd$(q,n)=1$, implies $n \neq q$. Thus, it is interesting to determine conditions requiring and consequences of $n<q$ and $q<n$.

A well known result of Nielsen \cite{nielsen2006} states that $N$ must consist of at least 9 different odd primes, i.e. $n$ must have at least 8 unique factors. At first glance, it would seem reasonable to guess that $q<n$. However, a quick consideration of Descartes famous ``spoof'' odd perfect number:
\begin{center}
$N=3^2 7^2 11^2 13^2 22021$
\end{center}
where if one pretends for a moment that 22021 is prime, and that $\sigma(22021) = 22022$, then $\sigma(N)=2N$. For this example, $q=22021$ and $q>n$. So it seems a plausible question to ask that if an odd perfect number exists, is it necessary that the special prime dominate the rest of the factors?

Our initial intuitions turn out to be correct. Dris proved in \cite{dris2012} that $k > 1 \Longrightarrow q < n$. Acquaah and Konyagin \cite{Acqu2012} later showed $k=1 \Longrightarrow q < (3N)^{1/3}$ from which it is immediate that $q < \sqrt{3} n$. (Their proof having been modified from Luca and Pomerance \cite{luca2010}.) In \cite{dris2015B}, Dagal and Dris, utilize Acquaah and Konyagin's results to show $q<n$ so long as $3 \nmid N$. In section 2, we utilize Neilsen's result to make a simple adjustment to Acquaah and Konyagin's argument to conclude $k=1 \Longrightarrow q < n$, which allows us to conclude $q<n$ (and mildly stronger results) for all odd perfect numbers.

Dris conjectured in \cite{dris2008} that $q^k < n$. In section 3, we endeavor to prove this conjecture adjusting Acquaah and Konyagin's argument even further. We start with a proof of the simplest case and show the argument can be massaged to conclude $k > 1 \Longrightarrow q^k < n$. However, limitations in the method prevent a complete proof without additional assumptions in the second and third case.

\section{Proof of $q<n$}

\begin{theorem} \label{premain}
Let $N = q^k n^2$ be an odd perfect number written in Eulerian form, then $q<n$.
\end{theorem}
\begin{proof}
As mentioned above, the case $k>1$ has been established, so we assume k=1. Rewrite $N$ in full as
\begin{center}
$N = q p^{2b} r_{1}^{2\beta_{1}} r_{2}^{2\beta_{2}}$, \ldots, $r_{j}^{2\beta_{j}}$.
\end{center}

\noindent Where $p$ is the unique prime whereby $q|\sigma(p^{2b})$, $r_i$ for $1 \leq i \leq j$, represent the rest of the primes dividing $N$. When convenient, we will truncate $N$ as
\begin{center}
$N = q \ p^{2b} \ r_{1}^{2\beta_{1}} \ w^{2}$
\end{center}

\noindent Let $c_i \geq 0$ be the integer whereby $p^{c_i} || \sigma(r_{i}^{2\beta_{i}})$ for $1 \leq i \leq j$. Where we give ``$||$'' its standard meaning that $p^{c_i} | \sigma(r_{i}^{2\beta_{i}})$, but $p^{c_i+1} \nmid \sigma(r_{i}^{2\beta_{i}})$. It is possible that $p^{c_i} = \sigma(r_{i}^{2\beta_{i}})$ for any particular $i$, but since we know $n$ has at least eight components, at least one of the $\sigma(r_{i}^{2\beta_{i}})$ has to have factors other than $p$. Thus, we may rewrite subscripts and assume:
\begin{center}
$p^{c_1} r_{2} | \sigma(r_{1}^{2\beta_{1}})$
\end{center}

\begin{center}
\textbf{Case 1: $p \nmid \sigma(q)$}
\end{center}

\EQ{2N = \sigma(N) = \sigma(q) \sigma(p^{2b}) \sigma(r_{1}^{2\beta_{1}}) \sigma(w^2)}

\noindent Observe that $p \nmid \sigma(q)$ implies $p^{2b-c_1} || \sigma(w^2)$. Thus,

\EQ{2N > (q+1) \ q \ (p^{c_1} r_{2}) \ (p^{2b-c_1})}

\noindent We now utilize the fact that $p^{2b} > \frac{2}{3} \sigma(p^{2b})$ and $r_2$ being an odd prime means $r_2 \geq 3$.

\EQ{2N > q^2 \ (3) \ \frac{2}{3} \sigma(p^{2b})}

\EQ{2N > q^2 \ (3) \ \frac{2}{3} \ q}

\EQ{N > q^3}

\noindent from which $q < n$ easily follows.

\begin{center}
\textbf{Case 2: $p|\sigma(q)$}
\end{center}

\noindent Let $p^{c_q} || \sigma(q)$. Let $u=\sigma(p^{2b})/q$. Since
\EQ{\sigma(p^{2b}) \equiv 1 \pmod{p}, \ \ \ \ q \equiv -1 \pmod{p}}

\noindent we know $u \equiv -1 \pmod{p}$. Since $u$ is odd, we know $u \neq p-1$, and thus $u \geq 2p-1$.

\noindent By construction, we have $p^{2b-c_q-c_1} || \sigma(w^2)$, which implies
\EQ{\sigma(w^2) \geq p^{2b-c_q-c_1}}

\noindent Observe now,
\EQ{p^{2b+1}-1 = (p-1)\sigma(p^{2b}) = (p-1)uq = (p-1)u\sigma(q)-(p-1)u}
\noindent Therefore, $(p-1)u \equiv 1 \pmod{p^{c_q}}$. Which implies $(p-1)u > p^{c_q}$.

\noindent Combining the last two inequalities yields,
\EQ{\sigma(w^2)(p-1)u > p^{2b-c_1} \ \ \ \Longrightarrow \ \ \ \sigma(w^2)u > \frac{p^{2b-c_1}}{p-1}}

\noindent This should be all we need:

\EQ{2N = \sigma(N) = \sigma(q) \sigma(p^{2b}) \sigma(r_{1}^{2\beta_{1}}) \sigma(w^2)}

\EQ{2N > (q+1) \ uq \ (p^{c_1} r_{2}) \ \sigma(w^2)}

\EQ{2N > q^2 \frac{p^{2b-c_1}}{p-1} \ p^{c_1} r_{2}}

\EQ{2N > q^2 \ r_{2} \ \frac{p^{2b}}{p-1}}

\noindent Again, we utilize $p^{2b} > \frac{2}{3} \sigma(p^{2b})$.

\EQ{2N > q^2 \ r_{2} \ \frac{2 \sigma(p^{2b})}{3(p-1)}}

\EQ{2N > q^2 \ r_{2} \ \frac{2uq}{3(p-1)}}

\noindent Recall, $u \geq 2p-1$ and again $r_2$ being an odd prime means $r_2 \geq 3$.

\EQ{2N > q^3 \ 3 \ (\frac{2}{3}) \ \frac{2p-1}{(p-1)}}

\EQ{N > 2q^3}

\noindent And again, we get $q < n$.

\end{proof}

There is nothing special about this method of proof and the result of $q < n$ compared to Acquaah and Kanyagin's estimate $q < \sqrt{3} n$. If one is prepared to do the bookkeeping to account for extra unaccounted for factors $r_3$, $r_4$, $r_5$, etc., one can estimate them as $r_3 \geq 5$, $r_4 \geq 11$, $r_5 \geq 13$, etc. to get $q < \frac{n}{(\sqrt{5*11*13 \ldots})}$.

\section{Partial Proof of Dris's Conjecture}

Assume $N=q^k n^2$ is an odd perfect number written in Eulerian form. The case k=1 is proven in Theorem \ref{premain}, so we assume $k \geq 5$. Our goal is to prove $q^k < n$ with as few assumptions as possible.  Rewrite $N$ in full as

\EQ{N = q^k p_1^{2b_1} \ldots p_s^{2b_s} r^{2\beta} w^{2}}

\noindent Where the $p_i$ are the primes whereby $q^{t_i}||\sigma(p_i^{2b_i})$, for integers $t_i \geq 0$ and $1 \leq i \leq s$. It is convenient to name another prime $r$ separate from the $p_i$'s and allow $w^2$ to represent the rest of the primes dividing $N$. We do not assume a priori that $r$ and $w$ necessarily exist and in such cases we simply take one or both to be one.

Let $c_i \geq 0$ be the integer whereby $p_i^{c_i} || \sigma(p_1^{2b_1}$ \ldots $p_s^{2b_s})$ for $1 \leq i \leq s$.

\begin{center}
\textbf{Case 1: $p_i \nmid \sigma(q^k)$ for each $i$}
\end{center}

Because $k$ is odd, $\sigma(q^k) = (1+q)(1+q^2+q^4$ \ldots $+q^{k-1})$. It is straight forward to show $(1+q^2+q^4$ \ldots $+q^{k-1})$ is coprime to its formal derivative, which makes the polynomial separable, and as such has no repeated roots modulo any prime. Thus any prime dividing $(1+q^2+q^4$ \ldots $+q^{k-1})$, divides at most once. Let $r$ be a prime dividing $\sigma(q^k)$ such that $r || \sigma(q^k)$. By assumption $r$ is not any of the $p_i$'s. Also note that we may assume $r \geq 7$. If $r = 1+q^2+q^4$ \ldots $+q^{k-1}$, then clearly $r \geq 7$. Otherwise, we may assume $r \equiv 1 \pmod{\frac{k+1}{2}}$, by virtue of the fact that $r$ divides a cyclotomic polynomial. For the smallest exponent, $k=5$, $\frac{k+1}{2} = 3$; and the smallest $1 \pmod{3}$ prime is 7.

\EQ{2N = \sig{N} = \{\sig{q^k}\} \{\sig{p_1^{2b_1} \ldots p_s^{2b_s}}\} \{\sig{r^{2\beta} w^{2}}\}}

\EQ{2N > \{q^k\} \ \ \{q^k p_1^{c_1} \ldots p_s^{c_s}\} \ \ \{p_1^{2b_1-c_1} \ldots p_s^{2b_s-c_s}\} \ \ \{r\}}

\noindent Note that the quantity in each brace on the right hand side is less than or equal to the quantity in the respective brace in the previous line, with the exception of $\{r\}$. Since $r$ divides $N$ an even number of times, but can only divide \sig{q^k} once, we know $r$ must divide either \sig{p_1^{2b_1} \ldots p_s^{2b_s}} or \sig{r^{2\beta} w^{2}}, in addition to our previous assumptions.

\EQ{2N > q^{2k} (r) p_1^{2b_1} \ldots p_s^{2b_s}}

\noindent We utilize the fact that $(\frac{p-1}{p}) \sigma(p^{2b}) < p^{2b}$ and that $r \geq 7$.

\EQ{2N > q^{2k} (7) \ \bprod{i=1}{s} (\frac{p_i - 1}{p_i}) \ \sig{p_1^{2b_1} \ldots p_s^{2b_s}}}

\EQ{2N > q^{2k} (7) \ \bprod{i=1}{s} (\frac{p_i - 1}{p_i}) \ q^k p_1^{c_1} \ldots p_s^{c_s}}

\EQ{2N > q^{3k} (7) \ \bprod{i=1}{s} (1 - \frac{1}{p_i})}

\noindent Next we use the well known result, if $0 < \theta_i < 1$ for $i=1$, \ldots, $s$, then
\begin{center}
$\bprod{i=1}{s} (1-\theta_i) \geq 1 - \bsum{i=1}{s} \theta_i$.
\end{center}

\EQ{2N > q^{3k} (7) \ (1 - \bsum{i=1}{s} \frac{1}{p_i})}

\noindent While not the first to prove $\bsum{p|N}{} \frac{1}{p} < ln(2)$, Cohen gives a simple proof of this fact in (\cite{cohen1978}).

\EQ{2N > q^{3k} (7) \ (1 - ln(2))}

\noindent Since $7(1 - ln(2)) > 2.14$, we have 

\EQ{q^{3k} < N = q^k n^2 \Longrightarrow q^k < n}

\noindent as required.

\begin{center}
\textbf{Case 2: $s=1$ and $p_1 |\sigma(q^k)$}
\end{center}

In Section 2, case 2, the result relied on being able to find two inequalities $u \geq 2p-1$ and $(p-1)u \geq p^{c_q}$. The former depending on $p$ being unique and the latter depending on $k=1$, which made $\sig{q^k} = q+1$. To give a full proof of Dris's conjecture using this methodology, these two obstacles will have to be overcome. In this case, with $s=1$, we get $p_1$ is unique.

We procede as before, let $c_{1q} \geq 0$ be the integer for which $p_{1}^{c_{1q}} || \sigma(q^k)$ and let $u = \frac{\sig{p_1^{2b_1}}}{q^k}$. We may again conclude $u \equiv -1 \pmod {p_1}$ and $u \geq 2p_1-1$, however,

\EQ{p_1^{2b+1}-1 = (p_1-1)\sigma(p_1^{2b}) = (p_1-1)uq^k = (p_1-1)u\sig{q^k}-(p_1-1)u\sig{q^{k-1}}}

\noindent allows us to, at best, conclude $(p_1-1)u\sig{q^{k-1}} \equiv 1 \pmod p_1^{c_{1q}}$; which seems to be unhelpful.

We push on, let $v = \frac{\sig{w^2}}{p_1^{2b_1-c_{1q}}}$

\EQ{2N = \sig{N} = \sig{q^k} \sig{p_1^{2b_1}} \sig{w^2}}

\EQ{2N > q^k \ uq^k \ v p_1^{2b_1-c_{1q}}}

\EQ{2N > q^{2k} \ uv \ p_1^{2b_1} p_1^{-c_{1q}}}

\EQ{2N > q^{2k} \ uv \ \frac{p_1-1}{p_1} \sig{p_1^{2b_1}} p_1^{-c_{1q}}}

\EQ{2N > q^{2k} \ uv \ \frac{p_1-1}{p_1} \ uq^k p_1^{-c_{1q}}}

\EQ{2N > q^{3k} \ u^2 v \ \frac{p_1-1}{p_1} \ p_1^{-c_{1q}}}

\EQ{2N > q^{3k} (2p_1-1)^2 v \ \frac{p_1-1}{p_1} \ p_1^{-c_{1q}}}

\EQ{N > q^{3k} (2p_1^2 - 4p_1 + \frac{5}{2} - \frac{1}{2p_1})v \ p_1^{-c_{1q}}}

\noindent We see Dris's conjecture follows immediately whenever $c_{1q} \leq 2$ or, with more knowledge about $N$, when $vp_1^2 > p_1^{c_{1q}}$. By Neilsen's result, $w$ must have at least 7 components, which makes the latter inequality seem quite likely. Since these are amongst the first theorems relating components of an odd perfect number, more research is clearly needed.

\begin{center}
\textbf{Case 3: $s>1$ and $p_i |\sigma(q^k)$ for at least one $i$}
\end{center}

Let $c_{iq} \geq 0$ be the integer for which $p_{i}^{c_{iq}} || \sigma(q^k)$ for $1 \leq i \leq s$. We begin as before,

\EQ{2N = \sig{N} = \sig{q^k} \sig{p_1^{2b_1} \ldots p_s^{2b_s}} \sig{w^{2}}}

\EQ{2N > q^k \ q^k p_1^{c_1} \ldots p_s^{c_s} \ p_1^{2b_1-c_1-c_{1q}} \ldots p_s^{2b_s-c_s-c_{sq}}}

\EQ{2N > q^{2k} \ p_1^{2b_1} \ldots p_s^{2b_s} \ p_1^{-c_{1q}} \ldots p_s^{-c_{sq}}}

\EQ{2N > q^{2k} \ \bprod{i=1}{s} (\frac{p_i-1}{p_i}) \sig{p_1^{2b_1} \ldots p_s^{2b_s}} \ p_1^{-c_{1q}} \ldots p_s^{-c_{sq}}}

\EQ{2N > q^{2k} \ \bprod{i=1}{s} (\frac{p_i-1}{p_i}) \ q^k p_1^{c_1} \ldots p_s^{c_s} \ p_1^{-c_{1q}} \ldots p_s^{-c_{sq}}}

\EQ{2N > q^{3k} \ \bprod{i=1}{s} (\frac{p_i-1}{p_i}) \ p_1^{c_1-c_{1q}} \ldots p_s^{c_s-c_{sq}}}

\noindent Using $\bprod{i=1}{s} (\frac{p_i-1}{p_i}) > 1-ln(2)$, we see now the result $q^k < n$ follows whenever

\EQ{p_1^{c_1} \ldots p_s^{c_s} > \frac{2}{1-ln(2)} p_1^{c_{1q}} \ldots p_s^{c_{sq}}}

\noindent We recap our results thus far in the following

\begin{theorem} \label{main}
Let $N = q^k n^2$ be an odd perfect number written in Eulerian form, then $q<n$. 

Write $N = q^k p_1^{2b_1}$, \ldots, $p_s^{2b_s} w^2$, where $q|\sig{p_i}$ for $1 \leq i \leq s$. Let $c_i$, $c_{iq} \geq 0$ be integers where $p_i^{c_i} || \sig{p_1^{2b_1}, \ldots, p_s^{2b_s}}$ and $p_i^{c_{iq}} || \sig{q^k}$ for $1 \leq i \leq s$.

If $k>1$ and one of the following holds:
\begin{enumerate}
\item $p_i \nmid \sig{q^k}$ for $1 \leq i \leq s$;
\item $s=1$, $p_1 |\sigma(q^k)$, and $c_{1q} \leq 2$;
\item $s>1$, and $p_i | \sig{q^k}$ for at least one $p_i$ and
\begin{center}
$p_1^{c_1} \ldots p_s^{c_s} \geq 7 p_1^{c_{1q}} \ldots p_s^{c_{sq}}$
\end{center}
\end{enumerate}
then $q^k < n$.
\end{theorem}

\section{Further Considerations}

The condition $p_1^{c_1} \ldots p_s^{c_s} \geq 7 p_1^{c_{1q}} \ldots p_s^{c_{sq}}$ in Theorem \ref{main} seems to suggest Dris's conjecture holds if
\begin{center}
\bprod{i \neq j}{} gcd$(\sig{p_i^{2b_i}p_j^{2b_j}},p_i^{2b_i}p_j^{2b_j}) >$ \bprod{i=1}{s} gcd$(\sig{q^k},p_i^{2b_i})$.
\end{center}
Again, at first glance, it seems nothing can be said about this situation, but Dandapat, Hunsucker, and Pomerance in \cite{dandapat1975}, Theorem 2 implies for each fixed $i$, there is a $j$ where
\begin{center}
gcd$(\sig{p_i^{2b_i}p_j^{2b_j}},p_i^{2b_i}p_j^{2b_j})>1$ for $i \neq j$
\end{center}
holds for most non-special components of $n$, not just for the restricted $p_i$'s as we have defined them.

The next obvious question may be for $N = q^k p^{2b} w^2$, an odd perfect number where $q$ is the special prime, $p$ is any other prime dividing $N$, and $w^2$ encompasses the rest of the components of $N$, in the same way we showed $q<n$, can we show $p<w$, $p^b < w$, or even $p^{2b} < w$?

\bibliography{references}{}

\begin{thebibliography}{1}

\bibitem{Acqu2012}
Peter Acquaah and Sergei Konyagin.
\newblock On prime factors of odd perfect numbers.
\newblock {\em Int. J. Number Theory}, 08:1537, 2012.
\newblock http://dx.doi.org/10.1142/S1793042112500935.

\bibitem{cohen1978}
G.~L. Cohen.
\newblock On odd perfect numbers.
\newblock {\em Fibonacci Quarterly}, 16:523--527, 1978.

\bibitem{dris2015B}
Keneth Adrian~P. Dagal and J.~A.~B. Dris.
\newblock The abundancy index of divisors of odd perfect numbers - part ii.
\newblock {\em (Preprint)}, Jun. 2015.
\newblock http://arxiv.org/abs/1309.0906.

\bibitem{dandapat1975}
G.~G. Dandapat, J.~L. Hunsucker, and Carl Pomerance.
\newblock Some new results on odd perfect numbers.
\newblock {\em Pacific Journal of Math.}, 57(2):359–364, 1975.

\bibitem{dris2008}
J.~A.~B. Dris.
\newblock Solving the odd perfect number problem: Some old and new approaches.
\newblock Master's thesis, De La Salle University, Manila, Philippines, 2008.

\bibitem{dris2012}
J.~A.~B. Dris.
\newblock The abundancy index of divisors of odd perfect numbers.
\newblock {\em J. Integer Seq.}, 15:Article 12.4.4, 2012.
\newblock https://cs.uwaterloo.ca/journals/JIS/VOL15/Dris/dris8.html.

\bibitem{luca2010}
Florian Luca and Carl Pomerance.
\newblock On the radical of a perfect number.
\newblock {\em New York J. Math}, 16:23--30, 2010.

\bibitem{nielsen2006}
Pace~P. Nielsen.
\newblock Odd perfect numbers have at least nine different prime factors.
\newblock {\em Math Comp.}, 76(160):2109–2126, 2007.

\end{thebibliography}
\bibliographystyle{plain}

\end{document}